\documentclass[11pt]{article}
\usepackage{amssymb}
\PassOptionsToPackage{obeyspaces}{url}
\usepackage[hidelinks]{hyperref}
\usepackage{bookmark}
\bookmarksetup{open,numbered,depth=5}
\usepackage{amsfonts}
\usepackage{amsmath}
\usepackage{dsfont} 
\usepackage{amsthm,enumitem}
\usepackage{pgf,tikz}
\usepackage{bm}
\usepackage{cancel}
\usepackage{comment}
\usetikzlibrary{calc}

\usepackage{mathtools}
\usepackage{xstring}

\author{
Zichao Dong\thanks{Extremal Combinatorics and Probability Group (ECOPRO), Institute for Basic Science (IBS), Daejeon, South Korea. Supported by the Institute for Basic Science (IBS-R029-C4). \texttt{$\{$zichao,hongliu$\}$@ibs.re.kr}. }
\and Jun Gao\thanks{Mathematics Institute, University of Warwick, Coventry, UK. Supported by ERC Advanced Grant 101020255 and the Institute for Basic Science (IBS-R029-C4).
\texttt{gj950211@gmail.com}.}
\and Hong Liu\footnotemark[1]
\and Minghui Ouyang\thanks{School of Mathematical Sciences, Peking University, Beijing 100871, China. Partially supported by a grant from New Cornerstone Science Foundation. \texttt{ouyangminghui1998@gmail.com}.}
\and Qiang Zhou\thanks{Academy of Mathematics and Systems Science, Chinese Academy of Sciences, Beijing, China, and University of Chinese Academy of Sciences, Beijing, China. \texttt{zhouqiang2021@amss.ac.cn}.}
}

\title{Set families: restricted distances via restricted intersections}

\date{}

\usepackage[nameinlink]{cleveref}

\newtheorem{theorem}{Theorem}[section]
\newtheorem{lemma}[theorem]{Lemma}
\newtheorem{corollary}[theorem]{Corollary}

\newtheorem{claim}[theorem]{Claim}
\newtheorem{proposition}[theorem]{Proposition}

\newtheorem{question}[theorem]{Question}
\newtheorem{conjecture}[theorem]{Conjecture}
\newtheorem*{example}{Example}

\newtheorem*{remark*}{Remark}

\newenvironment{poc}{\begin{proof}[Proof of Claim]}{\end{proof}}

\newcommand*{\eqdef}{\stackrel{\mbox{\normalfont\tiny def}}{=}} 
\newcommand*{\veps}{\varepsilon}                                

\newcommand*{\F}{\mathbb{F}}                                    
\DeclareMathOperator{\dist}{dist}                               
\newcommand*{\N}{\mathbb{N}}                                    
\newcommand*{\Z}{\mathbb{Z}}                                    
\newcommand*{\R}{\mathbb{R}}                                    
\renewcommand*{\S}{\mathbb{S}}                                    

\newcommand*{\cE}{\mathcal{E}}
\newcommand*{\cF}{\mathcal{F}}
\newcommand*{\cG}{\mathcal{G}}
\newcommand*{\cH}{\mathcal{H}}

\newcommand*{\cM}{\mathcal{M}}
\newcommand*{\cX}{\mathcal{X}}
\newcommand*{\cS}{\mathcal{S}}

\newcommand*{\De}{D_{\operatorname{even}}}

\allowdisplaybreaks[4]

\begin{document}

\maketitle

\begin{abstract}
    Denote by $f_D(n)$ the maximum size of a set family $\mathcal{F}$ on $[n] \stackrel{\mbox{\normalfont\tiny def}}{=} \{1, \dots, n\}$ with distance set $D$. That is, $|A \bigtriangleup B| \in D$ holds for every pair of distinct sets $A, B \in \mathcal{F}$. Kleitman's celebrated discrete isodiametric inequality states that $f_D(n)$ is maximized at Hamming balls of radius $d/2$ when $D = \{1, \dots, d\}$. We study the generalization where $D$ is a set of arithmetic progression and determine $f_D(n)$ asymptotically for all homogeneous $D$. In the special case when $D$ is an interval, our result confirms a conjecture of Huang, Klurman, and Pohoata. Moreover, we demonstrate a dichotomy in the growth of $f_D(n)$, showing linear growth in $n$ when $D$ is a non-homogeneous arithmetic progression. Different from previous combinatorial and spectral approaches, we deduce our results by converting the restricted distance problems to restricted intersection problems. 
    
    Our proof ideas can be adapted to prove upper bounds on $t$-distance sets in Hamming cubes (also known as binary $t$-codes), which has been extensively studied by algebraic combinatorialists community, improving previous bounds from polynomial methods and optimization approaches. 
    %
\end{abstract}

\section{Introduction} \label{sec:intro}

The study of set families with restricted intersections is fundamental in extremal set theory. A cornerstone result in this area is the Erd\H{o}s--Ko--Rado theorem \cite{EKR61}, which states that for any pair of positive integers $n$ and $k$ with $n \ge 2k$, if $\cF \subseteq \binom{[n]}{k}$ is a $k$-uniform intersecting family---meaning $A \cap B \ne \varnothing \, (\forall A, B \in \cF)$---then $|\cF| \le \binom{n-1}{k-1}$. This result can be generalized to $t$-intersecting families, where $|A \cap B| \ge t \, (\forall A, B \in \mathcal{F})$. In this case, it holds that $|\cF| \le \binom{n-t}{k-t}$ for sufficiently large $n$, and the Ahlswede--Khachatrian theorem \cite{AK97} provides a complete characterization of the maximum size of a $t$-intersecting family $\mathcal{F}$ for every $n, k$ and $t$. 

Denote by $A \bigtriangleup B$ the symmetric difference between $A$ and $B$, i.e.~$(A \setminus B) \cup (B \setminus A)$, and refer to $|A \bigtriangleup B|$ as the distance between them. By replacing the intersection restriction ``$|A \cap B| \ge t$'' with the distance restriction ``$|A \bigtriangleup B| \ge t$'', we naturally arrive at a problem which is central to coding theory. Intensive research efforts have been dedicated to bounding the size of set families under various distance constraints, including Singleton bound \cite{Sin64}, Plotkin bound \cite{Plo60}, Hamming bound, Elias--Bassalygo bound \cite{Bas65}, and Gilbert--Varshamov bound \cite{Gil52, Var57, JV04}. A comprehensive overview of these developments and their significance is available in the standard reference \cite{Lin99}.

A closely related problem, originally proposed by Erd\H{o}s \cite{Erd76}, considers families $\cF \subseteq 2^{[n]}$ with a forbidden intersection size $|A \cap B| \ne \ell \, (\forall A, B \in \cF)$. Erd\H{o}s conjectured that, with an offer of \$250 for a solution, if the ratio $\frac{\ell}{n}$ lies between $\bigl( \veps, \frac{1}{2} - \veps \bigr)$, then the family size $|\cF|$ is exponentially smaller than $2^n$. This conjecture was first solved by Frankl and R\"{o}dl \cite{FR87}. Keevash and Long \cite{KL17} later provided an alternative proof and gave a restricted distance result: if $\frac{d}{n} \in (\veps, 1-\veps)$ and $\cF \subseteq 2^{[n]}$ is a family with $|A \bigtriangleup B| \ne d \, (\forall A, B \in \cF)$, then $|\cF| \le 2^{(1-\delta)n}$. We also direct readers to \cite{KSZ22}.

While research on restricted intersections has found substantial applications in areas such as Ramsey theory and discrete geometry (see, e.g., \cite[Chapter 22]{frankl_tokushige}), comparatively little attention has been devoted to the restricted-distance families. Motivated by this gap, we study in this paper set families with restricted distances. For any distance set $D$, our objective is to investigate $f_D(n)$, the largest possible size of a $D$-distance family on $[n]$. 

\subsection{Generalizing Kleitman's isodiametric inequality}

Among all smooth convex bodies of fixed diameter in $\R^n$, the classical Euclidean isodiametric inequality (see \cite{EG15}) asserts that the ball attains the maximum volume. In the discrete setting, solving a conjecture of Erd\H{o}s, Kleitman \cite{kleitman} established an isodiametric result which states that
\[
f_D(n) = \begin{cases}
    \binom{n}{0} + \binom{n}{1} + \cdots + \binom{n}{t} \qquad &\text{if $D = [2t]$}, \\
    2\Bigl( \binom{n-1}{0} + \binom{n-1}{1} + \cdots + \binom{n-1}{t} \Bigr) \qquad &\text{if $D = [2t+1]$}. 
\end{cases}
\]
The bounds above are optimal, as witnessed by Hamming balls of appropriate radius. 

Kleitman's original proof relies on sophisticated set operations. Recently, Huang, Klurman, and Pohoata~\cite{huang_klurman_pohoata} discovered a smart linear algebraic proof of Kleitman's result. They also applied their technique to general intervals as distance sets and established that
\[
\bigl( 1 + o(1) \bigr) \cdot \frac{1}{\tbinom{t}{s}} \cdot \binom{n}{t-s} \le f_{\{2s+1, 2s+2, \dots, 2t\}}(n) \le \bigl( 1 + o(1) \bigr) \cdot \binom{n}{t-s}. 
\]
They conjectured that the lower bound gives the correct asymptotics. 

\smallskip

We study the same problem for even more general distance sets that are arithmetic progressions. An integer arithmetic progression is \emph{homogeneous} if its common difference divides its first term, and \emph{non-homogeneous} otherwise. Denote by $\De \eqdef \{d \in D \colon \text{$d$ is even}\}$ the set of all even integers in $D$.

\begin{theorem} \label{thm:homo}
    Let $d, s, t$ be positive integers with $1 \le s \le t$ and $D = \bigl\{ sd, (s+1)d, \dots, td \bigr\}$ be a homogeneous arithmetic progression. As $n \to \infty$, the followings hold. 
    \vspace{-0.5em}
    \begin{itemize}
        \item If $d$ is even, then
        \[
        f_D(n) = \bigl(1 + o(1)\bigr) \cdot \frac{1}{\binom{t}{|\De|}} \cdot \binom{2n/d}{|\De|} = \bigl(1 + o(1)\bigr) \cdot \frac{1}{\binom{t}{t-s+1}} \cdot \binom{2n/d}{t-s+1}. 
        \]

        \vspace{-0.75em}
        \item If $d$ is odd, then
        \[
        f_D(n) = \bigl(1 + o(1)\bigr) \cdot \frac{c}{\binom{\lfloor t/2 \rfloor}{|\De|}} \cdot \binom{n/d}{|\De|} = \bigl(1 + o(1)\bigr) \cdot \frac{c}{\binom{\lfloor t/2 \rfloor}{\lfloor t/2 \rfloor - \lceil s/2 \rceil + 1}} \cdot \binom{n/d}{\lfloor t/2 \rfloor - \lceil s/2 \rceil + 1}, 
        \]
        where $c = 1$ if $st$ is even and $c = 2$ if $st$ is odd. 
    \end{itemize}
    \vspace{-0.5em}
\end{theorem}

\Cref{thm:homo} provides asymptotically sharp bounds on $f_D(n)$ for every homogeneous arithmetic progression distance set $D$ with arbitrary common differences, which greatly generalizes previous results on interval distance sets.  In particular, by setting $d_{\ref{thm:homo}} \eqdef 1, \, s_{\ref{thm:homo}} \eqdef 2s + 1$, and $t_{\ref{thm:homo}} \eqdef 2t$, we confirm the above Huang--Klurman--Pohoata conjecture that $f_{\{2s+1, \dots, 2t\}}(n)=\bigl( 1 + o(1) \bigr) \cdot \frac{1}{\tbinom{t}{s}} \cdot \binom{n}{t-s}$. 

\smallskip

Somewhat surprisingly, the non-homogeneous arithmetic progressions exhibit a completely different behavior. In fact, we prove that $f_D(n)$ grows only linearly in this case. 

\begin{theorem} \label{thm:nonhomo}
    Given $n, d, s, t, a \in \mathbb{N}_+$ with $1 \le s \le t$ and $1 \le a < d$. For any non-homogeneous arithmetic progression $D = \bigl\{ sd+a, (s+1)d+a, \dots, td+a \bigr\}$, the followings hold. 
    \vspace{-0.5em}
    \begin{itemize}
        \item If $\De = \varnothing$, then $f_D(n) = 2$ when $n \ge \min(D)$, and $f_D(n) = 1$ otherwise. 
        \vspace{-0.5em}
        \item If $\De \ne \varnothing$, then $\bigl\lfloor \frac{2n}{\min(\De)} \bigr\rfloor \le f_D(n) \le n + 2$.
    \end{itemize}
\end{theorem}

Our approach differs from both the classical combinatorial shifting in~\cite{kleitman} and the linear algebraic method in~\cite{huang_klurman_pohoata}. For upper bounds, the main idea is to convert the restricted distance problems to restricted intersection problems. In \Cref{thm:homo}, given a family $\cF$ with prescribed distances, the key is to identify within $\cF$ a subfamily $\cF'$ which constitutes the majority of our family $\cF$ (\Cref{obs:aux_diff}) and is located on at most two slices of the hypercube (\Cref{obs:aux_size}). The difficult case is when the distance set $D = \bigl\{ sd, (s+1)d, \dots, td \bigr\}$ is an arithmetic progression with odd $dt$ and even $s$. To obtain the correct asymptotics for this case, we have to show that the two slices of $\cF'$ cannot be both large at the same time. To this end, we show that between the two slices of $\cF'$, there is a restricted (cross) intersection pattern and we need to analyze the structure of the common intersection of each slice. In \Cref{thm:nonhomo}, we apply a modular restricted intersection result. It is worth noting that the lower bound constructions towards~\Cref{thm:homo} involve additional difficulty. Inspired by~\cite{huang_klurman_pohoata}, we employ the classical R\"{o}dl nibble method. When $d, s, t$ are all odd, we have to build up a ``double almost design'' to achieve an additional factor $2$. 

\subsection{Bounding the size of binary \texorpdfstring{$t$}{t}-codes} 

Our method is versatile, and it applies to the study of binary $t$-codes as well. We begin with some terminologies. Let $\cM$ be a metric space endowed with distance function $\dist_{\cM}$. For finite $X \subseteq \cM$, denote by $D(X) \eqdef \bigl\{\dist_\cM(x, y) \colon x, y \in X, \, x \ne y\bigr\}$ the \emph{distance set} of $X$. For fixed $D$, we define 
\[
A(\cM, D) \eqdef \max \bigl\{ |X| : \text{$X \subseteq \cM$ such that $D(X) = D$} \bigr\}. 
\]
Following the terminology of coding theory, we call $X$ a $t$-\emph{code} if $|D(X)| = t$. Denote the maximum cardinality of a $t$-code as $A(\cM, t) \eqdef \max \bigl\{ A(\cM, D) : |D| = t \bigr\}$.

It is a fundamental challenge to establish good upper bounds on $A(\cM, t)$. Historically, Einhorn and Schoenberg \cite{einhorn_schoenberg_1,einhorn_schoenberg_2} initiated the study of $t$-codes in Euclidean space $\cM = \R^n$. As a classical application of the polynomial method, Larman, Rogers, and Seidel \cite{larman_rogers_seidel} (later slightly improved by Blokhuis \cite{blokhuis}) deduced that $\frac{n(n+1)}{2} \le A(\R^n, 2) \le \frac{(n+1)(n+2)}{2}$. The same problem for Euclidean sphere $\cM = \S^{n-1} \subseteq \R^n$ has also been intensively studied. Delsarte, Goethals, and Seidel \cite{DGS77} established $\frac{n(n+1)}{2} \le A(\S^{n-1}, 2) \le \frac{n(n+3)}{2}$. Related results and improvements can be found in a large series of papers \cite{musin2009,nozaki,BM11,musin_nozaki,barg_yu,yu,glazyrin_yu,musin2019,jiang_tidor_yao_zhang_zhao2023,liu_yu}.

We focus on such a problem concerning Hamming cubes. Let $\cH_n$ be the set of binary strings $\F_2^n$ endowed with the Hamming distance. (For any $\bm{x}, \bm{y} \in \F_2^n$, their \emph{Hamming distance} $\dist_{\cH_n}(\bm{x}, \bm{y})$ is the number of bits at which $\bm{x}$ differs from $\bm{y}$.) The study of binary codes has long been central to coding theory. Two-distance sets in $\cH_n$, binary $2$-codes, have received considerable attention as they connect to finite geometry. See \cite{calderbank_kantor} for an early survey and \cite{shi_honold_sole_qiu_wu_sepasdar} for recent work. Due to the extensive interest on $t$-codes in $\R^n$ and $\S^{n-1}$, it is natural to bound $A(\cH_n, t)$. Earliest such results date back to the influential work of Delsarte \cite{Del73} in 1973, where he established $A(\cH_n, t) \le \sum_{i = 0}^{t} \binom{n}{i}$. Most such results are obtained by either the polynomial method \cite{BBS83, Blo84} or the linear programming method~\cite{Del73, DGS77}, which reduce the problem to the study of spherical $t$-codes in $\mathbb{R}^n$. A folklore conjecture, communicated to us by Yu \cite{yu_private}, predicts that the maximum value is achieved by the $t$-set $\{2,4,\ldots, 2t\}$. 

\begin{conjecture}[Folklore]\label{conj:maxcode}
    For every sufficiently large $n$ (with respect to $t$), we have 
    \[
    A(\cH_n, t) = A \bigl( \cH_n, \{2, 4, \dots, 2t\} \bigr) = 
    \begin{cases}
        \binom{n}{0} + \binom{n}{2} + \dots + \binom{n}{t-2} + \binom{n}{t} \qquad &\text{if $t$ is even}, \\
        \binom{n}{1} + \binom{n}{3} + \dots + \binom{n}{t-2} + \binom{n}{t} \qquad &\text{if $t$ is odd}. 
    \end{cases}
    \]
\end{conjecture}

Towards this folklore conjecture, Nozaki and Shinohara \cite{NS10} refined the result of Delsarte \cite{Del73} to $A(\cH_n, t) \le \sum_{i:c_i > 0} \binom{n}{i}$, where $c_i$ are the coefficients defined by a polynomial related to $D$ in terms of the Krawtchouk polynomial basis. Barg and Musin \cite{BM11} further established an explicit formula $A(\cH_n, t) \le \binom{n}{t} + \sum_{i = 0}^{t-2} \binom{n}{i}$ under the condition $\sum_{d\in D}d \le \frac{tn}{2}$. For small values of $t$, partial results were proved in \cite{BM11} for $t = 2, 3, 4$. Later, the $t = 2$ case was resolved by Barg et al.~\cite{barg_glazyrin_kao_lai_tseng_yu}. 

\smallskip

Our next result establishes $A(\cH_n, D) \le A \bigl( \cH_n, \{2, 4, \dots, 2t\} \bigr)$ for any fixed $D$ and every sufficiently large $n$, which makes significant progress towards the first equality in \Cref{conj:maxcode}. We shall also prove that the second equality holds in \Cref{conj:maxcode}. 

\begin{theorem} \label{thm:exact}
    Let $D \subseteq \mathbb{N}_+$ be a $t$-set. If $D \ne \{2, 4, \dots, 2t\}$, then as $n \to \infty$ we have
    \[
    A(\cH_n, D) \le \Bigl(\frac{t}{t+1}+o(1)\Bigr) \cdot \binom{n}{t}. 
    \]
    When $n \ge 2t + 2$, we have 
    \[
    A \bigl( \cH_n, \{2, 4, \dots, 2t\} \bigr) = \begin{cases}
        \binom{n}{0} + \binom{n}{2} + \dots + \binom{n}{t-2} + \binom{n}{t} \qquad &\text{if $t$ is even}, \\
        \binom{n}{1} + \binom{n}{3} + \dots + \binom{n}{t-2} + \binom{n}{t} \qquad &\text{if $t$ is odd}. 
    \end{cases}
    \]
   
\end{theorem}

We remark that the bound $n \ge 2t + 2$ above is optimal. Indeed, when $n \le 2t + 1$, the behavior changes into $A \bigl( \cH_n, \{2, 4, \dots, 2t\} \bigr) = 2^{n-1}$. Here the lower bound is witnessed by the collection of all binary strings with odd (or even) numbers of $1$-s, and the upper bound follows from the observation that $\dist_{\cH_n}(\bm{x}, \bm{y})$ is odd when the number of $1$-s in $\bm{x}$ and $\bm{y}$ are of opposite parities. 

\Cref{thm:exact} can be viewed as a stability result, for it infers that the distance set $\{2, 4, \dots, 2t\}$ is the unique maximizer, showing that $A(\cH_n, D)$ is substantially smaller than $A \bigl( \cH_n, \{2, 4, \dots, 2t\} \bigr)$ for any other fixed $t$-set $D$. In contrast, due to the leading term $\binom{n}{t}$, the previous bounds could not perform well when $D$ differs from $\{2, 4, \dots, 2t\}$.

\vspace{-0.5em}

\paragraph{Paper organization.} We first prove \Cref{thm:homo}. The lower bounds are given in \Cref{sec:homo_lower} and the upper bounds are deduced in \Cref{sec:homo_upper}. We then prove \Cref{thm:nonhomo} in \Cref{sec:nonhomo}. In \Cref{sec:exact}, we prove~\Cref{thm:exact}. Finally, we include some further discussions and open problems in \Cref{sec:remark}. 

\section{Proof of the lower bounds in \texorpdfstring{\Cref{thm:homo}}{Theorem 1.1}} \label{sec:homo_lower}

In this section, we prove the lower bounds in \Cref{thm:homo}. Depending on the parity of $dst$, we divide the constructions for the lower bounds into two cases. 

\begin{theorem} \label{thm:homo_lower}
    Suppose $d, s, t$ are positive integers with $1 \le s \le t$. Let $D = \bigl\{ sd, (s+1)d, \dots, td \bigr\}$ be a homogeneous arithmetic progression. The following hold for sufficiently large $n$. 
    
    \vspace{-0.5em}
    \begin{enumerate}[label=(\roman*), ref=(\roman*)]
        \item \label{lower:even} If $dst$ is even, then there exists a $D$-distance set family $\cF \subseteq 2^{[n]}$ of size
        \[
        |\cF| = \begin{cases}
            \bigl(1 - o(1)\bigr) \cdot \dfrac{\binom{2n/d}{t-s+1}}{\binom{t}{t-s+1}} \qquad &\text{when $d$ is even}, \\
            \bigl(1 - o(1)\bigr) \cdot \dfrac{\binom{n/d}{\lfloor t/2 \rfloor - \lceil s/2 \rceil + 1}}{\binom{\lfloor t/2 \rfloor}{\lfloor t/2 \rfloor - \lceil s/2 \rceil + 1}} \qquad &\text{when $d$ is odd}. 
        \end{cases}
        \]

        \vspace{-0.75em}
        \item \label{lower:odd} If $dst$ is odd, then there exists a $D$-distance set family $\cF \subseteq 2^{[n]}$ of size
        \[
        |\cF| = \bigl(2 - o(1)\bigr) \cdot \frac{\binom{n/d}{\lfloor t/2 \rfloor - \lceil s/2 \rceil + 1}}{\binom{\lfloor t/2 \rfloor}{\lfloor t/2 \rfloor - \lceil s/2 \rceil + 1}}. 
        \]
    \end{enumerate}
\end{theorem}

Inspired by \cite[Theorem 3.2]{huang_klurman_pohoata}, we are going to apply the following lemma in our constructions to establish \Cref{thm:homo_lower}\ref{lower:even}. This lemma is known as ``the R\"{o}dl nibble'' \cite{rodl}. (See also \cite[Section 4.7]{alon_spencer}.) 

\begin{lemma} \label{lem:rodl}
    Let $s, t \in \mathbb{N}$ with $1 \le s \le t$. As $m \to \infty$, there exists a $t$-uniform set family $\cF \subseteq 2^{[m]}$ of size $|\cF| = \bigl( 1 - o(1) \bigr) \cdot \frac{\binom{m}{t-s+1}}{\binom{t}{t-s+1}}$ such that $|A \cap B| \le t - s$ holds for any pair of distinct $A, B \in \cF$. 
\end{lemma}

\begin{proof}[Proof of \Cref{thm:homo_lower}\ref{lower:even}]
    We first look at the case when $d$ is even. Write $m \eqdef \lfloor 2n/d \rfloor$. Thanks to \Cref{lem:rodl}, we are able to find a $t$-uniform family $\cF' \subseteq 2^{[m]}$ of size $|\cF'| = \bigl(1-o(1)\bigr) \cdot \frac{\binom{m}{t-s+1}}{\binom{t}{t-s+1}}$
    such that $|A' \cap B'| \le t - s$ holds for any pair of distinct $A', B' \in \cF'$. This implies that
    \[
    |A' \bigtriangleup B'| \in \bigl\{ 2t, 2t-2, \dots, 2t-2(t-s) \bigr\} = \bigl\{ 2s, 2s+2, \dots, 2t \bigr\}. 
    \]
    Since $d$ is even, we can replace each element in $[m]$ by a unique group of $d/2$ distinct elements in $[n]$. This replacement then turns $\cF'$ into a $D$-distance family $\cF \subset 2^{[n]}$ of desired size. 

    If $d$ is odd, we replace $D$ by $\De$ and the same construction as above gives us a $\De$-distance (hence $D$-distance) family $\cF \subset 2^{[n]}$ of desired size. That is, we begin with a $\lfloor t/2\rfloor$-uniform family of size $\bigl(1-o(1)\bigr) \cdot \frac{\binom{m}{\lfloor t/2\rfloor-\lceil s/2\rceil+1}}{\binom{\lfloor t/2\rfloor}{\lfloor t/2\rfloor-\lceil s/2\rceil+1}}$ and pairwise intersection size at most $\lfloor t/2\rfloor - \lceil s/2\rceil$, and then blowup each element in $[m]$ to a unique group of size $d$. (We have $\lfloor t/2\rfloor \ge \lceil s/2\rceil$ since $st$ is even.) 
\end{proof}

To prove \Cref{thm:homo_lower}\ref{lower:odd}, we need the following extension of~\Cref{lem:rodl}.   

\begin{lemma} \label{lem:rodlvariant}
    Let $s', t' \in \mathbb{N}_+$ with $1 \le s' \le t' + 1$. As $m \to \infty$, there exist two $t'$-uniform families $\cF_1, \cF_2 \subseteq 2^{[m]}$, each of size $\bigl(1-o(1)\bigr) \cdot \frac{\binom{m}{t'-s'+1}}{\binom{t'}{t'-s'+1}}$ such that
    \vspace{-0.5em}
    \begin{itemize}
        \item $|A \cap B| \le t' - s'$ holds for any pair of distinct $A, B \in \cF_1$ or $A, B \in \cF_2$, and 
        \vspace{-0.5em}
        \item $|A \cap B| \le t' - s' + 1$ holds for any pair of $A \in \cF_1, \, B \in \cF_2$. 
    \end{itemize}
    \vspace{-0.5em}
\end{lemma}

We remark that one can force $\cF_1$ and $\cF_2$ to be disjoint in \Cref{lem:rodlvariant}. In fact, the second condition directly shows the disjointness, provided that $s' > 1$. However, this is unnecessary for our application. 

\begin{proof}[Proof of \Cref{thm:homo_lower}\ref{lower:odd} assuming \Cref{lem:rodlvariant}]
    Let $s' \eqdef \lceil s/2 \rceil, \, t' \eqdef \lfloor t/2 \rfloor$. Conversely, this implies that $s = 2s' - 1, \, t = 2t' + 1$. Define the parameter $m \eqdef \lfloor n/d \rfloor - 1$. 
    
    Let $\cF_1, \cF_2 \subseteq 2^{[m]}$ be two $t'$-uniform families obtained from \Cref{lem:rodlvariant}. Append a new element $m + 1$ to each set in $\cF_2$. Then replace each element in $[m+1]$ by a unique group of $d$ distinct elements in $[n]$. Let the resulting families be $\cF_1', \cF_2'$. For any distinct $A, B \in \cF_1'$ or $A, B \in \cF_2'$, we have
    \[
    |A \bigtriangleup B| \in d \cdot \bigl\{ 2t', 2t'-2, \dots, 2t'-2(t'-s') \bigr\} = \bigl\{ (s+1)d, (s+2)d, \dots, (t-1)d \bigr\} \subseteq D. 
    \]
    Moreover, for any $A \in \cF_1', \, B \in \cF_2'$, we have
    \[
    |A \bigtriangleup B| \in d \cdot \bigl\{ 2t'+1, 2t'-1, \dots, 2t'+1-2(t'-s'+1) \bigr\} = \bigl\{ sd, (s+1)d, \dots, td \bigr\} = D. 
    \]
    Since $\cF_1'$ and $\cF_2'$ are disjoint (due to element $m + 1$), we see that $\cF_1' \cup \cF_2'$ is a $D$-distance family with
    \[
    |\cF_1' \cup \cF_2'| = |\cF_1'| + |\cF_2'| = |\cF_1| + |\cF_2| = \bigl(2 - o(1)\bigr) \cdot \frac{\binom{n/d}{\lfloor t/2 \rfloor - \lceil s/2 \rceil + 1}}{\binom{\lfloor t/2 \rfloor}{\lfloor t/2 \rfloor - \lceil s/2 \rceil + 1}}. \qedhere
    \]
\end{proof}

We are left to establish \Cref{lem:rodlvariant}. The R\"{o}dl nibble (\Cref{lem:rodl}) already implies the existence of $\cF_1$ and $\cF_2$ separately. However, it is difficult to control the intersections between sets in $\cF_1, \cF_2$. To achieve this, we need tools from combinatorial design theory. 

An $(N, q, r)$-\emph{design} is a family $\cS$ of $q$-subsets of $[N]$ in which every $r$-subset of $[N]$ appears as a subset of exactly one set in $\cS$. By considering the codegree of an $i$-subset for $i = 0, 1, \dots, r - 1$, an obvious necessary condition for the existence of an $(N, q, r)$-design is the \emph{divisibility condition}: 
\[
\text{$\tbinom{q-i}{r-i}$ divides $\tbinom{N-i}{r-i}$ for $i = 0, 1, \dots, r-1$}. 
\]
A profound result due to Keevash~\cite{keevash2014} states that this divisibility condition is sufficient for large $N$. 

\begin{theorem} \label{thm:design}
    Let $q, r$ be positive integers with $q \ge r$. For every sufficiently large positive integer $N$ satisfying the divisibility condition, an $(N, q, r)$-design exists. 
\end{theorem}

We remark that alternative and simpler proofs of \Cref{thm:design} can be found in \cite{glock_kuhn_lo_osthus, delcourt_postle, keevash2024}. 

\smallskip

To prove \Cref{lem:rodlvariant}, the idea is to begin with an $(N, t', t'-s'+2)$-design, and then extract $\cF_1, \cF_2$ successively from it. Recall that a hypergraph is \emph{linear} if any pair of its edges intersects in at most one vertex, and a \emph{matching} in a hypergraph is a set of vertex-disjoint hyperedges. 

Informally, the following technical result from \cite{alon_kim_spencer, kostochka_rodl} says that an almost perfect matching always exists in a uniform almost regular linear hypergraph. 

\begin{theorem} \label{thm:matching}
    For any positive integers $\ell \ge 3$ and $K$, there exists some $D_0 = D_0(\ell, K)$ such that the following property holds: Suppose $D \ge D_0$ and let $\cH = (V, \cE)$ be an $\ell$-uniform linear hypergraph with $D - K\sqrt{D \log D} \le \deg_{\cH}(v) \le D$ for all vertices $v \in V$. Then as $|V| \to \infty$, there exists a matching covering all but at most $O\bigl(|V|D^{-1/(\ell-1)}\log^{3/2}D\bigr)$ vertices.
\end{theorem}

Due to the almost regularity restriction $D - K\sqrt{D \log D} \le \deg_{\cH}(x) \le D$ in \Cref{thm:matching}, we have to find a design (using the deep result \Cref{thm:design}) instead of an almost design (via \Cref{lem:rodl}). 

\begin{proof}[Proof of \Cref{lem:rodlvariant}]
    When $s' = 1$, upon setting $\cF_1 = \cF_2 \eqdef \binom{[m]}{t'}$ we are done. If $s' = t' + 1$, then we can choose two single-set families supported on distinct elements $\cF_1 = [t'], \, \cF_2 = \{t'+1,\ldots,2t'\}$. 
    
    Assume $t'\ge s' \ge 2$ then. Pick the largest integer $N \le m$ satisfying $N \equiv q \pmod Q$, where 
    \[
    \textstyle Q \eqdef r! \cdot \prod\limits_{i=0}^{r-1} \binom{q-i}{r-i}, \qquad q \eqdef t', \qquad r \eqdef t' - s' + 2. 
    \]
    Since $s', t'$ (hence $q, r$) are fixed, we have $m - N = O(1)$. For $i = 0, 1, \dots, r - 1$, we compute
    \[
    (r-i)! \cdot \tbinom{N-i}{r-i} = (N-i) \cdots (N-r+1) \equiv (q-i) \cdots (q-r+1) = (r-i)! \cdot \tbinom{q-i}{r-i} \pmod{Q}. 
    \]
    Since $(r-i)! \cdot \binom{q-i}{r-i}$ divides $Q$, we have $\binom{q-i}{r-i}$ divides $\binom{N-i}{r-i}$. So, $N$ satisfies the divisibility condition. Due to \Cref{thm:design}, we can find an $(N, t', t'-s'+2)$-design $\mathcal{S}$ on $[N]$ when $m$ is sufficiently large. 

    Construct an $\ell$-uniform linear hypergraph $\cH = (V, \cE)$ of uniformity $\ell = \binom{t'}{t'-s'+1}$ as follows. Let $V$ be the set of $(t'-s'+1)$-subsets of $[N]$. For each $A \in \mathcal{S}$, include a hyperedge $\{B \in V \colon B \subseteq A\}$ into $\cE$. Since each $(t'-s'+2)$-subset of $[N]$ is contained in exactly one $t'$-subset in $\mathcal{S}$, the hypergraph $\cH$ is linear and $D$-regular, where $D = \frac{N-(t'-s'+1)}{t'-(t'-s'+1)}$. By \Cref{thm:matching}, since $D = \Omega(N)$ and $D_0 = O(1)$, there exists a matching $\mathcal{M}$ in the hypergraph $\mathcal{H}$ covering $\bigl(1 - o(1)\bigr)$-proportion of its vertices. Phrasing differently, this matching $\cM$ misses $o \Bigl( \binom{N}{t'-s'+1} \Bigr)$ vertices. Take $\cF_1$ to be the set of $t'$-sets in $\mathcal{S}$ which corresponds to the hyperedges in $\mathcal{M}$. 

    Consider the $\ell$-uniform linear hypergraph $\cH' = (V, \cE')$ for $\cE' \eqdef \cE \setminus \cM$. This is the hypergraph obtained by stripping a matching from the $D$-regular hypergraph $\mathcal{H}$. It follows that the degree of each vertex in $\mathcal{H}'$ is either $D-1$ or $D$. From \Cref{thm:matching} we deduce that $\mathcal{H}'$ also admits an almost perfect matching. Take $\cF_2$ to be the set of $t'$-sets in $\mathcal{S}$ corresponding to this matching. 

    For any pair of distinct sets $A, B \in \cF_1$ or $A, B \in \cF_2$, since their corresponding hyperedges in $\cH$ are vertex-disjoint, we have $|A \cap B| \le t' - s'$. Also, for any pair of $A \in \cF_1$ and $B \in \cF_2$, we have $|A \cap B| \le t'-s'+1$ as $\cH$ is linear. Since $m - N = O(1)$, the families $\cF_1, \cF_2$ meet all requirements in \Cref{lem:rodlvariant}, and hence the proof is complete. 
\end{proof}

\section{Proof of the upper bounds in \texorpdfstring{\Cref{thm:homo}}{Theorem 1.1}} \label{sec:homo_upper}

In this section, we prove the following result which implies the upper bounds in \Cref{thm:homo}. 

\begin{theorem} \label{thm:homo_upper}
    Suppose $d, s, t$ are positive integers with $1 \le s \le t$. Let $D = \bigl\{ sd, (s+1)d, \dots, td \bigr\}$ be a homogeneous arithmetic progression, and $\cF \subseteq 2^{[n]}$ be a $D$-distance family. As $n \to \infty$, we have
    \[
    |\cF| \le \begin{cases}
        \bigl(1 + o(1)\bigr) \cdot \prod\limits_{\ell \in \De} \frac{2n}{\ell} \qquad &\text{if $dst$ is even}, \\
        \bigl(2 + o(1)\bigr) \cdot \prod\limits_{\ell \in \De} \frac{2n}{\ell} \qquad &\text{if $dst$ is odd}. 
    \end{cases}
    \]
\end{theorem}

We remark that if $\De = \varnothing$ ($dst$ is odd and $s = t$), conventionally the empty product is $1$. 

Indeed, when $dst$ is odd, $\lfloor t/2 \rfloor = \frac{t-1}{2}, \, \lceil s/2 \rceil = \frac{s+1}{2}, \, |\De| = \lfloor t/2 \rfloor - \lceil s/2 \rceil + 1 = \frac{t-s}{2}$, and so
\[
\bigl(2 + o(1)\bigr) \cdot \prod_{\ell \in \De} \frac{2n}{\ell} = \bigl(2 + o(1)\bigr) \cdot \frac{(\frac{n}{d})^{\frac{t-s}{2}}}{\frac{s+1}{2} \cdot \frac{s+3}{2} \cdot \cdots \cdot \frac{t-1}{2}} = \bigl(2 + o(1)\bigr) \cdot \frac{\binom{n/d}{(t-s)/2}}{\binom{(t-1)/2}{(t-s)/2}},
\]
implying the upper bound in \Cref{thm:homo} in this case. Other cases are similar. 

\subsection{Distance families versus intersecting families}\label{sec:dist-vs-intersect}

A set family $\cF$ is $k$-\emph{uniform} if $|A| = k \, (\forall A \in \cF)$. For a set of non-negative integers $L$, we say that $\cF$ is an \emph{$L$-intersecting} family if $|A \cap B| \in L$ holds for each pair of distinct $A, B \in \cF$. A classical result concerning intersecting families is the Frankl--Wilson theorem below. 

\begin{theorem}[\cite{frankl_wilson}] \label{thm:frankl_wilson}
    Suppose $0 \le \ell_1 < \dots < \ell_r < k$ are integers and write $L \eqdef \{\ell_1, \dots, \ell_r\}$. If $\cF \subseteq 2^{[n]}$ is a $k$-uniform $L$-intersecting family, then $|\cF| \le \binom{n}{r}$. 
\end{theorem}

Let $\cF$ be a $D$-distance family. By flipping the belonging relation of a single number in $[n]$ to each set from $\cF$ simultaneously, we may assume without loss of generality that $\varnothing \in \cF$. It follows that $|A| = |A \bigtriangleup \varnothing| \in D$ holds for all $A \in \cF \setminus \{\varnothing\}$. Partition $\cF \setminus \varnothing$ into a disjoint union of $\cF_d$ for $d \in D$. From the fact $2|A \cap B| = |A| + |B| - |A \bigtriangleup B|$ we deduce that $\cF_d$ is an $L_d$-intersecting family, where the set of  intersection sizes is $L_d \eqdef d - \frac{\De}{2} = \bigl\{d - \frac{k}{2} \colon k \in \De\bigr\}$. After dropping negative integers in $L_d$, the Frankl--Wilson theorem (\Cref{thm:frankl_wilson}) directly shows the following crude upper bound. 

\begin{corollary} \label{coro:distance_weak}
    If $\cF \subseteq 2^{[n]}$ is a $D$-distance family, then $|\cF| \le |D| \cdot n^{|\De|}$. 
\end{corollary}

One can see that the order of magnitude in \Cref{coro:distance_weak} already matches that from the upper bounds in \Cref{thm:homo}, and the most interesting part in our main result is to find the sharp leading coefficient. The following result proved by Deza, Erd\H{o}s and Frankl \cite{deza_erdos_frankl} can be viewed as a refinement of the Frankl--Wilson theorem. This is crucial for us to establish the upper bounds in \Cref{thm:homo}. Instead of the original statement, here we state a slightly weaker version. 

\begin{theorem}[{\cite[Theorem 18.1]{frankl_tokushige}}] \label{thm:deza_erdos_frankl}
    Let $0 \le \ell_1 < \dots < \ell_r < k$ and $n \ge 2^k k^3$ be integers and write $L \eqdef \{\ell_1, \dots, \ell_r\}$. If $\cF \subseteq 2^{[n]}$ is a $k$-uniform $L$-intersecting family, then $|\cF| \le \prod_{i=1}^r \frac{n-\ell_i}{k-\ell_i}$. Moreover, if $|\cF| \ge 2^k k^2 n^{r-1}$, then there exists $C \subseteq [n]$ with $|C| = \ell_1$ such that $C$ is contained in every set in $\cF$. 
\end{theorem}

We also need the following corollary of~\Cref{thm:deza_erdos_frankl}.

\begin{lemma}\label{lem:uniform}
    If $\cG \subseteq 2^{[n]}$ is a $k$-uniform $D$-distance set family with $n \ge 2^k k^3$, then $|\cG| \le \prod\limits_{\ell \in \De} \frac{2n}{\ell}$. 
\end{lemma}

\begin{proof}
    For distinct $A, B \in \cG$, since $|A \bigtriangleup B| = |A| + |B| - 2|A \cap B|$, we see that $\cG$ is an $L$-intersecting family, where $L \eqdef \bigl( k - \frac{\De}{2} \bigr) \cap \Z_{\ge 0}$. It then follows from \Cref{thm:deza_erdos_frankl} that 
    \[
    |\cG| \le \prod_{\ell \in L} \frac{n-\ell}{k-\ell} \le \prod_{\ell \in \De} \frac{2n}{\ell}. \qedhere
    \]
\end{proof}

\subsection{Proof of \texorpdfstring{\Cref{thm:homo_upper}}{Theorem 3.1}}

The $\De = \varnothing$ case is easy. We shall discuss it separately at the beginning of \Cref{sec:nonhomo}. 

\smallskip

Here and after we assume that $\De \ne \varnothing$. By passing to a subprogression of $D$ if necessary, we may assume that there is a pair of sets $\cF$ realizing the maximum distance $td$ in $D$. By flipping the belonging relation of a single element in $[n]$ to each set from $\cF$ simultaneously and then relabeling, we may further assume that $\varnothing \in \cF$ and $[td] \in \cF$. Introduce the auxiliary subfamily
\[
\cF' \eqdef \bigl\{ A \in \cF \colon \bigl| A \setminus [td] \bigr| \ge d^*/2 \bigr\}, \quad \text{where $d^* \eqdef \max(\De)$}. 
\]
To upper bound $|\cF|$, we are to show that $|\cF \setminus \cF'|$ is small; and to upper bound $|\cF'|$ by analyzing its intersection pattern and applying the Frankl--Wilson and the Deza--Erd\H{o}s--Frankl theorems. 

\begin{claim} \label{obs:aux_diff}
    As $n \to \infty$, we have $|\cF \setminus \cF'| = O \bigl( n^{|\De|-1} \bigr)$. 
\end{claim}

\begin{poc}
    For each $S \subseteq [td]$, define $\cF_S \eqdef \bigl\{ A \in \cF \setminus \cF' \colon A \cap [td] = S \bigr\}$. It suffices to show that $|\cF_S| =  O\bigl(n^{|\De|-1}\bigr)$. For distinct $A, B \in \cF_S$, note that $|A \bigtriangleup B| \le \bigl|A \setminus [td]\bigr| + \bigl|B \setminus [td]\bigr| < d^*$. Thus, $\cF_S$ is a $D \setminus \{d^*\}$-distance family. The desired bound on $|\cF_S|$ then follows from~\Cref{coro:distance_weak}. 
\end{poc}

It remains to upper bound the size of $\cF'$. 

\begin{claim} \label{obs:aux_size}
For any $A\in \cF'$, we have
\[
|A| = \begin{cases}
    td \qquad &\text{if $td$ is even}, \\
    \text{$td$ or $(t-1)d$} \qquad &\text{if $td$ is odd}. 
\end{cases}
\]
\end{claim}

\begin{poc}
    Write $x \eqdef \bigl| A \cap [td] \bigr|$ and $y \eqdef \bigl| A \setminus [td] \bigr| \ge d^*/2$. Since $\cF'$ is a $D$-distance family, 
    \begin{align*}
        y + (td-x) = \bigl| A \setminus [td] \bigr| + \bigl| [td] \setminus A \bigr| = \bigl| A \bigtriangleup [td] \bigr| \le td \implies y \le x. 
    \end{align*}
    It follows that $|A| = x + y \ge 2y \ge d^*$, and so $|A| \ge td$ if $td$ is even while $|A| \ge (t-1)d$ if $td$ is odd. On the other hand, $|A| \le td$ as $\varnothing \in \cF$, concluding the claim. 
\end{poc}

If $td$ is even, then $d^* = td$, and \Cref{obs:aux_size} says that $\cF'$ is a $td$-uniform $D$-distance set family. So, from \Cref{lem:uniform} and \Cref{obs:aux_diff} we deduce that
\begin{equation} \label{eq:td_even}
    |\cF| = |\cF'| + |\cF \setminus \cF'| \le \prod_{\ell \in \De} \frac{2n}{\ell} + O\bigl(n^{|\De|-1}\bigr) = \bigl(1+o(1)\bigr) \cdot \prod_{\ell \in \De} \frac{2n}{\ell}. 
\end{equation}

If $td$ is odd, then $d^* = (t-1)d$. Due to \Cref{obs:aux_size}, we are able to partition $\cF'$ into a $td$-uniform family $\cF'_1\eqdef \{ A \in \cF' : |A| = td \}$ and a $(t-1)d$-uniform family $\cF'_2\eqdef \{ A \in \cF' \colon |A| = (t-1)d \}$.
It then follows from \Cref{lem:uniform} and \Cref{obs:aux_diff} that
\begin{equation} \label{eq:td_odd}
    |\cF| = |\cF'_1| + |\cF'_2| + |\cF \setminus \cF'| \le \bigl(2+o(1)\bigr) \cdot \prod_{\ell \in \De} \frac{2n}{\ell}. 
\end{equation}

By combining \eqref{eq:td_even} and \eqref{eq:td_odd}, \Cref{thm:homo_upper} is proved in all but one case. We are left with improving the upper bound on $|\cF|$ by a factor of $2$ when $d, t$ are odd and $s$ is even. In this case, 
\[
\De = \bigl\{ sd, (s+2)d, \dots, (t-1)d \bigr\}. 
\]
To achieve this, we need to analyze $\cF'_1, \cF'_2$ carefully and utilize the structural part of \Cref{thm:deza_erdos_frankl}. 

For distinct $A, B\in \cF'$, we compute the possible sizes of their intersection as follows:
\[
|A \cap B| = \frac{|A|+|B|-|A \bigtriangleup B|}{2} \in \begin{cases}
    td - \frac{\De}{2} \qquad &\text{if $A, B \in \cF'_1$}; \\
    (t-1)d - \frac{\De}{2} \qquad &\text{if $A, B \in \cF'_2$};\\
    (t-1)d - \frac{\De}{2} \qquad &\text{if $A \in \cF'_1$ and $B \in \cF'_2$}. 
\end{cases}
\]
Denote 
\begin{align*}
    L_1 &\eqdef td - \tfrac{\De}{2} = \Bigl\{ \tfrac{(t+1)d}{2}, \tfrac{(t+3)d}{2}, \dots, \tfrac{(2t-s)d}{2} \Bigr\}, \\
    L_2 &\eqdef (t-1)d - \tfrac{\De}{2} = \Bigl\{ \tfrac{(t-1)d}{2}, \tfrac{(t+1)d}{2}, \dots, \tfrac{(2t-s-2)d}{2} \Bigr\}. 
\end{align*}
It follows that $\cF'_1$ is an $L_1$-intersecting family and $\cF'_2$ is an $L_2$-intersecting family. 

For any set family $\cG$, we denote its \emph{center} as $\bigcap \cG \eqdef \bigcap_{A \in \cG} A$. If the center of $\cF'_1$ is of size smaller than $\min \bigl( td - \frac{\De}{2} \bigr) = \frac{(t+1)d}{2}$, then \Cref{thm:deza_erdos_frankl} implies that $|\cF'_1| = O\bigl( n^{|\De|-1} \bigr)$, and hence
\[
|\cF| = |\cF'_1| + |\cF'_2| + |\cF \setminus \cF'| \le \bigl(1+o(1)\bigr) \cdot \prod_{\ell \in \De} \frac{2n}{\ell}. 
\]
We may then assume $\bigl| \bigcap \cF'_1 \bigr| \ge \frac{(t+1)d}{2}$ and similarly $\bigl| \bigcap \cF'_2 \bigr| \ge \min \bigl( (t-1)d - \frac{\De}{2} \bigr) = \frac{(t-1)d}{2}$. 

Arbitrarily fix a $\frac{(t+1)d}{2}$-element subset $X_1 \subseteq \bigcap \cF'_1$ and a $\frac{(t-1)d}{2}$-element subset $X_2 \subseteq \bigcap \cF'_2$. 

\smallskip

\noindent\textbf{Case 1.} $X_2 \subseteq X_1$. 

\smallskip

Denote $Y \eqdef X_1 \setminus X_2$ and observe that $|Y| = \frac{(t+1)d}{2} - \frac{(t-1)d}{2} = d$. For each $S \subseteq Y$, we write
\[
\cF_{2, S} \eqdef \{A \in \cF'_2 \colon A \cap Y = S\}
\]
and thus partition $\cF'_2 = \bigcup_{S \subseteq Y} \cF_{2, S}$ into $2^{|Y|} = 2^d$ subfamilies. For each non-empty $S \subseteq Y$ and any distinct $A, B \in \cF_{2, S}$, we have $|A \cap B| \ge |X_2| + |S| > \frac{(t-1)d}{2}$. Then $\cF_{2, S}$ is a $(t-1)d$-uniform and $L_2 \setminus \bigl\{ \frac{(t-1)d}{2} \bigr\}$-intersecting family, and so from \Cref{thm:frankl_wilson} we deduce that $|\cF_{2, S}| = O\bigl( n^{|\De|-1} \bigr)$. 

Define $\cF_{2, \varnothing}^+ \eqdef \{A \cup Y \colon A \in \cF_{2, \varnothing}\}$. Then, crucially, we observe that the union $\cF'_1 \cup \cF_{2, \varnothing}^+$ forms a $td$-uniform $L_1$-intersecting family, since sets between $\cF'_1, \cF'_2$ are $L_2$-intersecting and $L_1 = L_2 + |Y|$. 
This implies $|\cF'_1 \cup \cF_{2, \varnothing}^+| \le\prod\limits_{\ell \in \De} \frac{2n}{\ell}$ by \Cref{thm:deza_erdos_frankl}. 

We claim that $\cF'_1$ and $\cF_{2, \varnothing}^+$ are disjoint. Suppose there is some $A \cup Y \in \cF'_1 \cap \cF_{2, \varnothing}^+$. Again, sets between $\cF'_1$ and $\cF'_2$ have intersection size at most $\max(L_2) = \frac{(2t-s-2)d}{2} \le (t-2)d$ as $s\ge 2$ is even. However, $A \cup Y \in \cF'_1$ and $A \in \cF_{2, \varnothing} \subseteq \cF_2'$ have intersection size $|A| = (t-1)d$, a contradiction. So,
\[
|\cF| = |\cF'_1| + |\cF'_2| + |\cF \setminus \cF'| = |\cF'_1 \cup \cF_{2, \varnothing}^+| + \sum_{\varnothing \ne S \subseteq Y} |\cF_{2, S}| + |\cF \setminus \cF'| \le \bigl(1+o(1)\bigr) \cdot \prod\limits_{\ell \in \De} \frac{2n}{\ell}. 
\]

\smallskip

\noindent\textbf{Case 2.} $X_2 \nsubseteq X_1$. 

\smallskip

Pick some element $x \in X_2$ and set $Z \in \cF'_1$ such that $x \notin Z$. For each $T \subseteq Z$, we write
\[
\cF_{2, T} \eqdef \{A \in \cF'_2 \colon A \cap Z = T\}
\]
and partition $\cF'_2$ into $2^{|Z|} = 2^{td}$ subfamilies. Recall that for any $A \in \cF'_2$, we have $|A \cap Z| \in L_2$, and so $\cF_{2, T} = \varnothing$ whenever $|T| < \min(L_2) = \frac{(t-1)d}{2}$. Fix an arbitrary $T \subseteq Z$ with $|T| \ge \frac{(t-1)d}{2}$. For every distinct $A, B \in \cF_{2, T}$, observe that $T \cup \{x\} \subseteq A \cap B$, hence $|A \cap B| \ge \frac{(t+1)d}{2}$. It follows that $\cF_{2, T}$ is $(t-1)d$-uniform and $L_2 \setminus \bigl\{ \frac{(t-1)d}{2} \bigr\}$-intersecting, and so $|\cF_{2, T}| = O\bigl( n^{|\De|-1} \bigr)$ by \Cref{thm:frankl_wilson}. So, 
\[
|\cF| = |\cF'_1| + |\cF'_2| + |\cF \setminus \cF'| = |\cF'_1| + \sum_{T \subseteq Z} |\cF_{2, T}| + |\cF \setminus \cF'| \le \bigl(1+o(1)\bigr) \cdot \prod\limits_{\ell \in \De} \frac{2n}{\ell}. 
\]

We thus conclude the case when $s$ is even and $d, t$ are odd. The proof of \Cref{thm:homo_upper} is done. 

\section{Proof of \texorpdfstring{\Cref{thm:nonhomo}}{Theorem 1.2}} \label{sec:nonhomo}

The $\De = \varnothing$ case is easy, and is already included in the previous work due to Huang, Klurman and Pohoata \cite[Section 3]{huang_klurman_pohoata}. In fact, the $D$-distance family $\bigl\{ \varnothing, [sd+a] \bigr\}$ shows $f_D(n) \ge 2$, and from
\[
|A \bigtriangleup B| + |B \bigtriangleup C| + |C \bigtriangleup A| = 2 \bigl( |A \cup B \cup C| - |A \cap B \cap C| \bigr)
\]
we deduce that $f_D(n) \le 2$. It follows that $f_D(n) = 2$ whenever $\De = \varnothing$ and $n \ge \min (D)$. 

\smallskip

Hereafter we focus on the $\De \ne \varnothing$ case. Set $u \eqdef \frac{\min(\De)}{2}$. Then any collection of $\bigl\lfloor \frac{n}{u} \bigr\rfloor$ disjoint $u$-subsets of $[n]$ is a $D$-distance family, and so $f_D(n) \ge \bigl\lfloor \frac{n}{u} \bigr\rfloor = \bigl\lfloor \frac{2n}{\min(\De)} \bigr\rfloor$. The lower bound follows. 

To deduce the upper bound, we need the following result concerning modular intersecting families due to Babai, Frankl, Kutin, and \v{S}tefankovi\v{c}. 

\begin{theorem}[{\cite[Theorem 1.1]{babai_frankl_kutin_stefankovic}}] \label{thm:intersecting_mod}
    Suppose $p$ is a prime, $q \eqdef p^k \, (k \in \N_+)$, and $L \subseteq \{0, 1, \dots, q-1\}$. Assume $X$ is an $n$-element ground set and $A_1, \dots, A_m$ are subsets of $X$ with the following property: 
    \vspace{-0.5em}
    \begin{itemize}
        \item For any $i, j \in [m]$, the modular cardinality $|A_i \cap A_j| \pmod q \in L$ if and only if $i \ne j$. 
    \end{itemize}
    \vspace{-0.5em}
    Set $D \eqdef 2^{|L|-1}$. Then $m \le \binom{n}{D} + \binom{n}{D-1} + \dots + \binom{n}{0}$. 
\end{theorem}

Let $\cF$ be a $D$-difference family and assume without loss of generality that $\varnothing \in \cF$. Consider the family $\cF' \eqdef \cF \setminus \{\varnothing\}$. Then $|A| = |A \bigtriangleup \varnothing| \in D$ holds for any $A \in \cF'$. For any sets $A, B$, note that
\begin{equation} \label{eq:sym_diff}
    2|A \cap B| = |A| + |B| - |A \bigtriangleup B|. 
\end{equation}
Since $a < d$, we can find a prime power $q = p^k$ that divides $d$ but not $a$. Here $p$ is a prime.
\vspace{-0.5em}
\begin{itemize}
    \item If $p = 2$, then the fact $\De \ne \varnothing$ implies that $a$ is even and $k \ge 2$. For each $A \in \cF'$, we have $|A| \equiv a \pmod {2^{k-1}}$. For distinct $A, B \in \cF'$, from \eqref{eq:sym_diff} we deduce that $|A \cap B| \equiv \frac{a}{2} \pmod {2^{k-1}}$. Notice that $a\not\equiv \frac{a}{2} \pmod {2^{k-1}}$ as $q$ does not divide $a$. Thus, \Cref{thm:intersecting_mod} applied to $2^{k-1}$ and $L = \bigl\{ \frac{a}{2} \bigr\}$ implies that $|\cF'| \le n + 1$, and so $|\cF| \le n + 2$. 
    \vspace{-0.5em}
    \item If $p \ge 3$, then $q$ is odd. For every $A \in \cF'$, we have $|A| \equiv a \pmod q$. For distinct $A, B \in \cF'$, from \eqref{eq:sym_diff} we deduce that $2|A \cap B| \equiv a \pmod q$ hence $|A \cap B| \equiv \frac{a(q+1)}{2} \pmod q$. Again, notice that $a \not\equiv \frac{a(q+1)}{2} \pmod q$ since $q$ does not divide $a$. Therefore, \Cref{thm:intersecting_mod} applied to $q$ and $L = \bigl\{ \frac{a(q+1)}{2} \bigr\}$ implies that $|\cF'| \le n + 1$, and so $|\cF| \le n + 2$. 
\end{itemize}
\vspace{-0.5em}
By combining the cases above, the proof of \Cref{thm:nonhomo} is complete. 

\begin{remark*}
    Later, Xu and Yip pointed out to us that, using a more refined analysis, they managed to improve the $n + 2$ upper bound in~\Cref{thm:nonhomo} to $n + 1$ (see~\cite[Proposition 6.6]{xu_yip}). Indeed, this can be further improved to $n$ for certain classes of $D$ (for instance, this holds whenever there exists a prime $p$ dividing $d$ but not $a$). Observe that the improved bound is optimal when $\min(\De) = 2$. However, when $\min(\De) \ge 4$, it remains open whether the bound can be improved to anything below $n$. See~\Cref{question:1} in \Cref{sec:remark} for further discussions. 
\end{remark*}

\section{Proof of \texorpdfstring{\Cref{thm:exact}}{Theorem 1.4}} \label{sec:exact}

Fix a $t$-set $D \subseteq \mathbb{N}_+$ with $D \ne \{2, 4, \dots, 2t\}$ and take $n$ be sufficiently large in terms of $t$. Thanks to \Cref{coro:distance_weak}, it suffices to prove the special case that $D$ consists of even numbers. Let $\cF$ be a $D$-distance family. Write $d \eqdef \max(D)$ and assume without loss of generality that $\varnothing \in \cF, \, [d] \in \cF$. Then for any set $A\in \cF\setminus\{\varnothing\}$, we have $|A|\in D$. 

\subsection{Bounding \texorpdfstring{$A(\cH_n, D)$ away from $A \bigl( \cH_n, \{2, 4, \dots, 2t\} \bigr)$}{A(H, D) away from A(H, \{2,4,...,2t\})}}

If $d \ge 4t^2$, then we write $\cF(z) \eqdef \{A \in \cF \colon |A| = z\}$ for each $z \in D$. Due to \Cref{lem:uniform}, 
\[
|\cF| = |\{\varnothing\}| + \sum_{z \in D} |\cF(z)| \le 1 + \sum_{z \in D} \biggl( \prod_{\ell \in D} \frac{2n}{\ell} \biggr) \le 1 + t \cdot \frac{n^{t-1}}{(t-1)!} \cdot \frac{n}{2t^2} \le \Bigl( \frac{1}{2} + o(1) \Bigr) \cdot \binom{n}{t}. 
\]

We may then assume $d < 4t^2$. Similar to the proof of \Cref{thm:homo_upper}, we denote
\[
\cF' \eqdef \bigl\{ A \in \cF \colon \bigl| A \setminus [d] \bigr| \ge d/2 \bigr\}
\]
and bound from above the sizes of $\cF'$ and $\cF \setminus \cF'$ separately. 
\vspace{-0.5em}
\begin{itemize}
    \item For any $A \in \cF'$, write $x \eqdef \bigl| A \cap [d] \bigr|$ and $y \eqdef \bigl| A \setminus [d] \bigr| \ge d/2$. Then $\bigl|A \bigtriangleup [d]\bigr| = d - x + y \le d$, and so $x \ge y \ge d/2$, which implies that $|A| = x + y \ge d$. Combining this with $|A| = |A \bigtriangleup \varnothing| \le d$, we see that $|A| = d$. So, $\cF'$ is $d$-uniform. Thus, \Cref{lem:uniform}, along with the facts that $D$ is a $t$-set consisting of even numbers and $D \ne \{2, 4, \dots, 2t\}$, implies $|\cF'| \le \prod\limits_{\ell \in D} \frac{2n}{\ell} \le \frac{n^t}{(t-1)!(t+1)}$. 
    \vspace{-1em}
    \item For any $S \subseteq [d]$, denote $\cF_S \eqdef \bigl\{ A \in \cF \setminus \cF' \colon A \cap [d] = S \bigr\}$. For any pair of distinct $A, B \in \cF_S$, from $|A \bigtriangleup B| \le \bigl|A \setminus [d]\bigr| + \bigl|B \setminus [d]\bigr| < d$ we deduce that $\cF_S$ is a $D \setminus \{d\}$-distance family. Then \Cref{coro:distance_weak} implies $|\cF_S| \le (t-1) \cdot n^{t-1}$. It follows that $|\cF \setminus \cF'| < 2^{4t^2}(t-1) \cdot n^{t-1}$. 
\end{itemize}
\vspace{-0.5em}
Combining the estimates above, we conclude the first part of~\Cref{thm:exact} by estimating
\[
|\cF| = |\cF'| + |\cF \setminus \cF'| \le \frac{n^t}{(t-1)!(t+1)} + 2^{4t^2} (t-1) \cdot n^{t-1} \le \Bigl( \frac{t}{t+1} + o(1) \Bigr) \cdot \binom{n}{t}. 
\]

\subsection{The unique maximizer \texorpdfstring{$D_0 = \{2, 4, \dots, 2t\}$}{D = \{2,4,...,2t\}}}

Let $n \ge 2t + 2$ and $D_0 = \{2, 4, \dots, 2t\}$. Then $A \bigl( \cH_n, \{2, 4, \dots, 2t\} \bigr) = f_{D_0}(n)$. We need to prove that
\[
f_{D_0}(n) = \begin{cases}
    \binom{n}{0} + \binom{n}{2} + \dots + \binom{n}{t-2} + \binom{n}{t} \qquad &\text{if $t$ is even}, \\
    \binom{n}{1} + \binom{n}{3} + \dots + \binom{n}{t-2} + \binom{n}{t} \qquad &\text{if $t$ is odd}. 
\end{cases}
\]
When $t$ is even (resp.~odd), the bound is achieved by subsets of $[n]$ of even (resp.~odd) size up to $t$.

Let $\cF \subset 2^{[n]}$ be a $D_0$-distance set family. We are going to construct an auxiliary graph in which we may view $\cF$ as an independent set. The following corollary of Cauchy's Interlacing Theorem was discovered earlier by Cvetkovi\'c. This provides a useful technique to upper bound the independence number $\alpha(G)$ of a graph $G$. (A proof can be found after \cite[Corollary 2.5]{huang_klurman_pohoata}.) 

\begin{proposition}[\cite{cvetkovic}] \label{prop:Cvetkovic}
    Let $G = \bigl( [m], E \bigr)$ be an $m$-vertex graph and $M$ be a pseudo-adjacency matrix of $G$, i.e., a symmetric $m \times m$ matrix with $M_{ij} = 0$ whenever $ij \not\in E$. Let $n_{\le 0}(M), \, n_{\ge 0}(M)$ be the number of non-positive, non-negative eigenvalues of $M$. Then $\alpha(G) \le \min \bigl\{ n_{\le 0}(M), n_{\ge 0}(M) \bigr\}$. 
\end{proposition}

If the distance $|X \bigtriangleup Y|$ is even, then we observe that $|X|$ and $|Y|$ are of the same parity. So, $\cF$ consists solely of even-sized subsets or odd-sized subsets. By taking the symmetric difference with $\{1\}$ altogether, we may assume without loss of generality that $\cF \subseteq \cX \eqdef \bigl\{ X \subseteq 2^{[n]} \colon \text{$|X|$ is even} \bigr\}$. 

For each integer $k \ge 0$, construct an $|\cX| \times |\cX|$ matrix $M_k$ whose entries are either $0$ or $1$ with
\[
\text{$(M_k)_{X, Y} = 1$ if and only if $X, Y \in \cX$ and $|X \bigtriangleup Y| = 2k$}. 
\]
Let $V \eqdef \R^{\cX}$ be the vector space indexed by $\cX$. For any $S \subseteq [n]$, set $\bm{v}_S \in V$ with $(\bm{v}_S)_X \eqdef (-1)^{|X \cap S|}$. 

\begin{proposition} \label{prop:eigen}
    We have $\bm{v}_S = \bm{v}_{[n] \setminus S} \, (\forall S \subseteq [n])$ and the vectors $\bigl\{ \bm{v}_S \colon S \subseteq [n] \bigr\}$ form an orthogonal basis of $V$. Moreover, the matrix $M_k$ has an eigenvector $\bm{v}_S$ of eigenvalue $\sum_{i=0}^{|S|}(-1)^{i}\binom{|S|}{i}\binom{n-|S|}{2k-i}$. 
\end{proposition}

\begin{proof}
    From $X \in \cX$ we deduce that
    \[
    |X \cap S| \equiv |X| - |X \cap S| = \bigl| X \cap \bigl([n] \setminus S\bigr) \bigr| \pmod 2, 
    \]
    and so $\bm{v}_S = \bm{v}_{[n] \setminus S}$. For any pair of $S, S' \subseteq [n]$ with $S' \ne S$ and $S' \ne [n] \setminus S$, we have
    \begin{align*}
        \langle \bm{v}_S, \bm{v}_{S'} \rangle &= \sum_{X \in \cX} (-1)^{|X \cap S|} \cdot (-1)^{|X \cap S'|} = \sum_{X \in \cX} (-1)^{|X \cap \left( S \bigtriangleup S' \right)|} \\
        &= \sum_{k=0}^{\infty} \biggl( \sum_{X \subseteq [n] \colon |X| = 2k} (-1)^{|X \cap (S \bigtriangleup S')|} \biggr) \\
        &= \sum_{k=0}^{\infty} \Biggl( \sum_{i=0}^{2k} (-1)^i \binom{|S \bigtriangleup S'|}{i} \binom{n - |S \bigtriangleup S'|}{2k-i} \Biggr) \\
        &= \sum_{i=0}^{\infty} \Biggl( (-1)^i \binom{|S \bigtriangleup S'|}{i} \cdot \sum_{k = \lceil i/2 \rceil}^{\infty} \binom{n-|S \bigtriangleup S'|}{2k-i} \Biggr) \\
        &\overset{(\text{a})}{=} \sum_{i=0}^{\infty} \Biggl( (-1)^i \binom{|S \bigtriangleup S'|}{i} \cdot \frac{2^{n - |S \bigtriangleup S'|}}{2} \Biggr) \overset{(\text{b})}{=} 0, 
    \end{align*}
    where (a) follows from $S' \ne [n] \setminus S$ together with $\sum_{i=0}^{\infty} (-1)^i \binom{m}{i} = 0$, and (b) follows from $S \ne S'$. Since $\bigl| \bigl\{ \bm{v}_S \colon S \subseteq [n] \bigr\} \bigr| = 2^{n-1} = |\cX|$, the set $\bigl\{ \bm{v}_S \colon S \subseteq [n] \bigr\}$ forms an orthogonal basis of $V$. 
    
    Abbreviate $[n] \setminus Z$ as $Z^c$ for any $Z \subseteq [n]$. For $X, Y \subseteq [n]$, observe that 
    \[
    |Y \cap S| = |X \cap S| + |X^{c} \cap Y \cap S| - |X \cap Y^{c} \cap S|. 
    \]
    Think about $\bm{v}_S$ as a column vector. It then follows that
    \begin{align*}
        (M_k \cdot \bm{v}_S)_{X} &= \sum_{Y \in \cX} (M_k)_{X, Y} \cdot (\bm{v}_S)_Y = \sum_{Y \subseteq [n] \colon |X \bigtriangleup Y| = 2k} (-1)^{|Y \cap S|} \\
        &= (-1)^{|X \cap S|} \sum_{Y \subseteq [n] \colon |X \bigtriangleup Y| = 2k} (-1)^{|X^{c} \cap Y \cap S| + |X \cap Y^{c} \cap S|} \\
        &\overset{(\text{c})}{=} (-1)^{|X \cap S|} \sum_{Y \subseteq [n] \colon |Y| = 2k} (-1)^{|Y \cap S| + |\varnothing|} \\
        &= (\bm{v}_S)_X \cdot \sum_{i=0}^{|S|} (-1)^i \binom{|S|}{i}\binom{n-|S|}{2k-i}, 
    \end{align*}
    where at (c) we flip for every element of $X$ the belonging between it and $X, Y$ altogether, hence we assume $X = \varnothing$ in the summation. This is valid because such a flipping does not affect the summand $(-1)^{|X^{c} \cap Y \cap S| + |X \cap Y^{c} \cap S|}$. So, $\bm{v}_S$ is an eigenvector of $M_k$ whose eigenvalue is exactly 
    \[
    \sum_{i=0}^{|S|}(-1)^{i}\binom{|S|}{i}\binom{n-|S|}{2k-i}. \qedhere
    \]
\end{proof}

Construct a graph $G$ on vertex set $\cX$, where we put an edge between distinct $S, T \in \cX$ if and only if $|S \bigtriangleup T| \ge 2t+2$. The $D_0$-distance family $\cF$ corresponds to an independent set in $G$. Define a matrix $M \eqdef \sum_{k \ge t+1} \binom{k-1}{t} M_k$. (The infinite range of $k$ is unproblematic because $M_k = 0$ whenever $k > n/2$.) Then $M$ is a pseudo-adjacency matrix of $G$ if we identify $\cX$ and $\bigl[ |\cX| \bigr]$. 

For every integer $k$, denote by $\bigl[x^k\bigr] f(x)$ the coefficient of the $x^k$ term in the Laurent expansion of a meromorphic function $f(x)$ in the $|x| > 1$ region on the complex plane. For instance, we have the identity $\binom{k-1}{t} = (-1)^{t+1} \bigl[x^{-2k}\bigr] \Bigl(\frac{1}{(1-x^2)^{t+1}}\Bigr)$ because of $\binom{-(t+1)}{k} = (-1)^k \binom{k+t}{t}$ and the expansion
\vspace{-0.5em}
\[
\frac{1}{(1-x^2)^{t+1}} = (-1)^{t+1} x^{-2(t+1)} (1 - x^{-2})^{-(t+1)} = (-1)^{t+1}x^{-2(t+1)} \sum_{k = 0}^\infty \binom{k+t}{t} x^{-2k}. 
\]
By \Cref{prop:eigen}, $\bm{v}_S$ is an eigenvector of each $M_k$, so the eigenvalue of $M$ with respect to it is
\begin{align*}
    \lambda_S &= \sum_{k \ge t+1} \Biggl( \binom{k-1}{t} \sum_{i=0}^{|S|}(-1)^{i}\binom{|S|}{i}\binom{n-|S|}{2k-i} \Biggr) \\
    &= (-1)^{t+1} \sum_{k \ge t+1} \bigl[x^{-2k}\bigr] \biggl(\frac{1}{(1-x^2)^{t+1}}\biggr) \cdot \bigl[x^{2k}\bigr] \Bigl( (1-x)^{|S|} (1+x)^{n-|S|} \Bigr) \\
    &= (-1)^{t+1} \bigl[x^0\bigr] \biggl( \frac{(1-x)^{|S|} (1+x)^{n-|S|}}{(1-x^2)^{t+1}} \biggr). 
\end{align*}

\vspace{-1.5em}
\begin{itemize}
    \item If $|S| \in \{t+1, \dots, n-t-1\}$, then the denominator of $\frac{(1-x)^{|S|} (1+x)^{n-|S|}}{(1-x^2)^{t+1}}$ gets canceled out and hence the quotient is a polynomial in $x$. So, the constant term is exactly the evaluation of the function at $x = 0$, which is $1$. It follows that $\lambda_S=(-1)^{t+1}$. 
    \vspace{-0.5em}
    \item If $|S| \in \{1, \dots, t\}$, then we compute
        \begin{align*}
            \bigl[x^0\bigr] \frac{(1-x)^{|S|} (1+x)^{n-|S|}}{(1-x^2)^{t+1}} &= \bigl[x^0\bigr] \frac{(1+x)^{n-|S|-t-1}}{(1-x)^{t+1-|S|}} \\
            &= \bigl[x^0\bigr] \biggl( (-1)^{t+1-|S|} \cdot \frac{1}{x^{t+1-|S|}} \cdot \frac{(1+x)^{n-|S|-t-1}}{(1-\frac{1}{x})^{t+1-|S|}} \biggr) \\
            &= (-1)^{t+1-|S|} \bigl[x^{t+1-|S|}\bigr] \biggl( (1+x)^{n-|S|-t-1} \Bigl(1+\frac{1}{x}+\cdots\Bigr)^{t+1-|S|} \biggr). 
        \end{align*}
        It follows from $n \ge 2t + 2$ that $n - |S| - t - 1 \ge t + 1 - |S|$. So, every coefficient is non-zero, and hence the sign of $\lambda_S$ is $(-1)^{t+1} \cdot (-1)^{t+1-|S|} = (-1)^{|S|}$. 
    \vspace{-0.5em}
    \item If $|S| \in \{n-t, \dots, n\}$, then $\bm{v}_S = \bm{v}_{[n] \setminus S}$ (part of \Cref{prop:eigen}) implies that $\lambda_S = \lambda_{[n] \setminus S}$. It follows from the previous case that the sign of $\lambda_S$ is $(-1)^{n-|S|}$. 
\end{itemize}
\vspace{-0.5em}
According to \Cref{prop:Cvetkovic}, 
\vspace{-0.25em}
\begin{itemize}
    \item if $t$ is even, then $\alpha(G) \le n_{\ge 0}(M) = \binom{n}{0} + \binom{n}{2} + \cdots + \binom{n}{t}$; 
    \vspace{-0.25em}
    \item if $t$ is odd, then $\alpha(G) \le n_{\le 0}(M) = \binom{n}{1} + \binom{n}{3} + \cdots + \binom{n}{t}$. 
\end{itemize}
\vspace{-0.25em}
We thus conclude the proof of \Cref{thm:exact} by noticing that $|\cF| \le \alpha(G)$. 

\section{Concluding remarks} \label{sec:remark}
In this paper, we determine $f_D(n)$ asymptotically for all homogeneous arithmetic progress $D$. 
When $D$ is non-homogeneous, \Cref{thm:nonhomo} shows that $f_D(n)$ grows linearly in $n$. We are unable to derive an asymptotically tight bound on $f_D(n)$. It remains an interesting problem to determine the leading coefficient here. We propose the following easier problem. 

\begin{question} \label{question:1}
    Suppose $D$ is a non-homogeneous arithmetic progression with $\min(D) \ge 3$. Is it true that for every sufficiently large $n$, we have $f_D(n) \le (1-\veps)n$ for some $\veps > 0$? 
\end{question}

One natural relaxation of the above question is to impose additional uniformity constraints on the family $\cF$ and make it an $L$-intersecting family for some $L$, where many well-developed methods and theories can be applied. A prototypical example of such a relaxed question is the following. 

\begin{question} \label{question:2}
    Let $p \ge 3$ be a prime with $p \nmid r$. Suppose $\cF \subseteq 2^{[n]}$ is a set family satisfying: 
    \vspace{-0.5em}
    \begin{itemize}
        \item $|A| \equiv r \not\equiv 0 \pmod p$ holds for any $A \in \cF$, and
        \vspace{-0.5em}
        \item $|A \cap B| \equiv 0 \pmod p$ holds for any pair of distinct $A, B \in \cF$. 
    \end{itemize}
    \vspace{-0.5em}
    Is it true that there exists some $\veps > 0$ such that $|\cF| \le (1-\varepsilon)n$ all sufficiently large $n$? 
\end{question}

In \Cref{question:2}, the family $\cF$ is an $\{2r, 2r+p, 2r+2p, \cdots\}$-distance family. When $r = 1$, this is evidently false as witnessed by all singletons in $[n]$. For a general $r$, Shengtong Zhang provided the following counterexample: let $q \equiv r \pmod{p}$ be a sufficiently large prime (power) and consider the finite projective plane $\operatorname{PG}(2, \F_q)$. It contains $q^2+q+1$ points, $q^2+q+1$ lines (each of size $q+1$), where any two lines intersect in exactly one point. By adding $p-1$ dummy points simultaneously into each line, one obtains an $\{r \bmod p\}$-intersecting family of size $q^2+q+1$ on $[q^2+q+p]$. 

Nevertheless, this example does not refute \Cref{question:1}, for the distance set $D = \{2q\}$ therein grows with $n$. One may further hope that, if $D$ is a fixed non-homogeneous arithmetic progression with $\min D \ge 3$, then $f_D(n) \le \frac{n}{2}$. However, the following shows that this is false when $p = 3$. 

\begin{example}
    There exists a $5$-uniform $\{0, 3\}$-intersecting family on $[n]$ of size $5 \lfloor \frac{n}{9} \rfloor$. 
\end{example}

\begin{proof}
    For $i = 1, 2, \dots, \bigl\lfloor \frac{n}{9} \bigr\rfloor$, write $\widetilde{i} \eqdef 9(i-1)$ and consider
    \[
    \widetilde{i} + \{1,2,3,4,5\}, \quad \widetilde{i} + \{1,2,3,6,7\}, \quad \widetilde{i} + \{1,2,3,8,9\}, \quad \widetilde{i} + \{1,2,4,6,8\}, \quad \widetilde{i} + \{1,3,4,6,9\}. 
    \]
    Then the $5 \bigl\lfloor \frac{n}{9} \bigr\rfloor$ many $5$-element subsets from above form a $\{0, 3\}$-intersection family. 
\end{proof}

This construction shows that the lower bounds in \Cref{thm:nonhomo} are not asymptotically correct for some specific $D$, notably $D = \{4, 10\}$. To be specific, it implies that $f_{\{4, 10\}}(n) \ge 5 \bigl\lfloor \frac{n}{9} \bigr\rfloor$. 

\smallskip

Regarding~\Cref{conj:maxcode}, our approach fails to deduce the first equality for all $D$. This is because in one of our tools, the Deza--Erd\H{o}s--Frankl theorem (\Cref{thm:deza_erdos_frankl}), one needs $n$ to be at least exponential in the uniformity $k$. \Cref{sec:exact} in fact establishes a slightly stronger result than \Cref{thm:exact} that $A(\cH_n, D) \le A \bigl( \cH_n, \{2, 4, \dots, 2t\} \bigr)$ holds for all $D$ and large $n$ with $|D| = t$ and $\max D \le \frac{\log n}{2}$. Towards a complete resolution of \Cref{conj:maxcode}, we suspect that essential new ideas are required. In particular, when $t = 1$, the existence of an $n$-by-$n$ Hadamard matrix implies that $A \bigl( \cH_n, \{2\} \bigr) = A \bigl( \cH_n, \{\frac{n}{2}\} \bigr) = n$. This shows that, in constrast to the case when $n$ is sufficiently large compared to $D$, the set $D_0 = \{2, 4, \dots, 2t\}$ is not necessarily a unique maximizer of $A(\cH_n, D)$ in general, and hence our proof cannot be easily adapted to resolve the conjecture. 

\section*{Acknowledgments}

We are grateful to Boris Bukh for bringing the Deza--Erd\H{o}s--Frankl theorem to our attention. We also benefited from discussions with Andrey Kupavskii, Cosmin Pohoata, Wei-Hsuan Yu, and Shengtong Zhang. This work was initiated at the $2^{\text{nd}}$ ECOPRO Student Research Program in the summer of 2024. MO would like to thank ECOPRO for hosting. Part of this work was done during a visit of ZD and MO to Shandong University. They are thankful to Guanghui Wang for his support.

\bibliographystyle{plain}
{
\bibliography{general_kleitman}}

\end{document}